\documentclass [10pt]{article}
\usepackage{amssymb,amsmath,amstext}

\usepackage{secmods} 

\input epsf

  \def\Quote #1\par #2\par #3\par
	{\begin{flushright}
	   \quotefont  \baselineskip 12pt \lineskip 10pt #1\par
	   \vskip 5pt                                            
	   \authorfont \baselineskip  8pt \lineskip  5pt #2\par
	   \vskip 1pt
	   \workfont   \baselineskip  3pt \lineskip  3pt #3\par
         \end{flushright} 
	 \vskip 25pt}

  \def\mtx#1#2{\renewcommand{\arraystretch}{1.2}%
      \left(\! \begin{array}{#1}#2\end{array}\! \right)}

\def\nextline{$ $\newline}

\def\definedas{\defined}

\def\T{^T\!}

\def\defined{\buildrel{\scriptscriptstyle\triangle}\over=}

\def\m{\phantom-}

\def\Re{\mathbb{R}}

\def\minimize#1{{\displaystyle\minim_{#1}}}

\def\blackslug{\hbox{\hskip 1pt \vrule width 4pt height 6pt depth 1.5pt
  \hskip 1pt}}

\def\minim{\mathop{\hbox{\rm minimize}}}

\def\rank{\mathop{\hbox{\rm rank}}}

\def\re{\mathop{\hbox{\rm re}}}

\def\half  {{\textstyle\frac12}}
\def\third {{\textstyle\frac13}}
\def\twothirds{{\textstyle\frac23}}

\def\inv{^{\raise 0.9pt\hbox{$\scriptscriptstyle -$}1}} 

\def\spose#1{\hbox to 0pt{#1\hss}}

 \def\text #1{\hbox{\quad#1\quad}}

\def\nthinsp{\mskip -2   mu}

\def\subplus {_{\scriptscriptstyle +}}

\def\kp#1{_{k{\raise 0.7pt\hbox{$\scriptscriptstyle +$}}#1}} 
\def\km#1{_{k{\raise 0.7pt\hbox{$\scriptscriptstyle -$}}#1}} 

\def\jp#1{_{j{\raise 0.7pt\hbox{$\scriptscriptstyle +$}}#1}} 
\def\jm#1{_{j{\raise 0.7pt\hbox{$\scriptscriptstyle -$}}#1}} 

\def\ip#1{_{i{\raise 0.7pt\hbox{$\scriptscriptstyle +$}}#1}} 
\def\im#1{_{i{\raise 0.7pt\hbox{$\scriptscriptstyle -$}}#1}} 

\def\superstar{^{\raise 0.5pt\hbox{$\nthinsp *$}}}
\def\SUPERSTAR{^{\raise 0.5pt\hbox{$*$}}}

\def\lamstar  {\lambda\superstar}
\def\lamstarT {\lambda^{\raise 0.5pt\hbox{$\nthinsp\ast$}T}}

\def\pstar{p\superstar}

\def\vstar{v\superstar}

\def\xstar{x\superstar}

\def\lamhat{\widehat\lambda}
\def\lamtilde{\widetilde\lambda}

\ifx\varGamma\undefined%
 \def\varGamma{{\mathit\Gamma}}

 \def\varLambda{{\mathit\Lambda}}

\else
\fi

\def\Lambdait{\varLambda}

\def\Ascr{{\mathcal A}}

\def\Dscr{{\mathcal D}}

\def\Iscr{{\mathcal I}}
\def\Jscr{{\mathcal J}}

\def\Vscr{{\mathcal V}}
\def\Wscr{{\mathcal W}}

\def\Atilde{\widetilde A}

\def\Dtilde{\widetilde D}

\def\hbar{\skew{4.2}\bar h}

\def\xbar{\skew{2.8}\bar x}
\def\xhat{\skew{2.8}\widehat x}
\def\xtilde{\skew3\widetilde x}

 \def\sigmahat{\widehat\sigma}

 \def\ceil(#1){\lceil#1\rceil}    
 \def\floor(#1){\lfloor#1\rfloor}

 \def\re(#1,#2){\mathop{\mbox{re}}(#1,#2)}    
 \def\ae(#1,#2){\mathop{\mbox{ae}}(#1,#2)}

 \def\thinmtx#1#2{\left(\! \begin{array}{#1}#2\end{array}\! \right)}

 \def\Iscr{{\cal I}}
 \def\Escr{{\cal E}}

 \def\pstarT{p^{\raise 0.2pt\hbox{$\scriptstyle\nthinsp\ast$}T}}
 \def\qstarT{q^{\raise 0.2pt\hbox{$\scriptstyle\nthinsp\ast$}T}}
 \def\sstarT{s^{\raise 0.2pt\hbox{$\scriptstyle\nthinsp\ast$}T}}

 \def\xstarT{x^{\raise 0.5pt\hbox{$\nthinsp *$}T}}

\def\vv#1{\verb|#1|}

\def\m{\phantom-}

\def\AstarT{A^{{\raise 0.5pt\hbox{$*$}}T}}
\def\JstarT{J^{{\raise 0.5pt\hbox{$*$}}T}}
\def\lamstarT{\lambda^{{\raise 0.5pt\hbox{$\nthinsp *$}}T}}
\def\Lamstar{\Lambdait^{{\raise 0.5pt\hbox{$\nthinsp *$}}}}
\def\Lamhatstar{{\widehat\Lambdait}^{{\raise 0.5pt\hbox{$\nthinsp *$}}}}
\def\lamhatstar{{\widehat\lambda}^{{\raise 0.5pt\hbox{$\nthinsp *$}}}}
\def\lamstarT{\lambda^{{\raise 0.5pt\hbox{$\nthinsp *$}}T}}

\def\lamstaraT{\lambda_a^{\raise 0.5pt\hbox{$\nthinsp *$}T}}
\def\lamstari{\lambda_i^{{\raise 0.5pt\hbox{$\nthinsp *$}}}}

\def\astari{a_i^{{\raise 0.5pt\hbox{$\nthinsp *$}}}}
\def\cstari{c_i^{{\raise 0.5pt\hbox{$\nthinsp *$}}}}
\def\xstari{x_i^{{\raise 0.5pt\hbox{$\nthinsp *$}}}}
\def\xstarT{x^{{\raise 0.5pt\hbox{$\nthinsp *$}}T}}
\def\zstari{z_i^{{\raise 0.5pt\hbox{$\nthinsp *$}}}}
\def\zstarT{z^{{\raise 0.5pt\hbox{$\nthinsp *$}}T}}
\def\Hstari{H_i^{{\raise 0.5pt\hbox{$*$}}}}

\def\+{\subplus}

\def\minimize#1{{\displaystyle\minim_{#1}}}

\def\Aact{\overset{\;\,=}{\vphantom{a}\smash{A}}}
\def\bact{\overset{\,=}{\vphantom{a}\smash{b}}}

\def\lamact{\overset{=}{\vphantom{a}\smash{\lambda}}}
\def\lamactT{{\vphantom{\rule{1pt}{5.5pt}}\smash{\overset{=}{\vphantom{a}\smash{\lambda}}}}^T}

\def\eacteps{\overset{\;=}{\rule{0pt}{0.7ex}\smash{e}}_{\!\epsilon}}

\def\Wscrhat{{\widehat\Wscr}}
\def\Wscrtilde{{\widetilde\Wscr}}

\def\lamhat{{\widehat\lambda}}

\def\subwscrhat{_{\scriptscriptstyle{\Wscrhat}}}

\def\subwscr{_{\scriptscriptstyle\Wscr}}

\def\subwscrone{_{\scriptscriptstyle{\Wscr_1}}}
\def\subwscrtwo{_{\scriptscriptstyle{\Wscr_2}}}
\def\subwscrthree{_{\scriptscriptstyle{\Wscr_3}}}
\def\subwscreps{_{\scriptscriptstyle{\Wscr_{\epsilon}}}}
\def\subwscrhat{_{\scriptscriptstyle{\Wscrhat}}}
\def\subyscreps{_{\scriptscriptstyle{\Yscr_{\epsilon}}}}

\def\supermorespacestar{^{\raise 0.5pt\hbox{$*$}}}
\def\Working{W}
\def\Workinghat{{\widehat\Working}}

\def\WorkingstarT
{\Working^{{\raise 0.5pt\hbox{$\nthinsp *$}}T}}

\def\WorkingstarplusT
{\Working\subplus^{{\raise 0.5pt\hbox{$\nthinsp *$}}T}}

\def\lamstarT{\lambda^{{\raise 0.5pt\hbox{$\nthinsp *$}}T}}
\def\lamstari{\lambda_i^{{\raise 0.5pt\hbox{$\nthinsp *$}}}}
\def\lamtilde{{\widetilde\lambda}}
\def\lamstarplus
{\lambda\subplus^{{\raise 0.5pt\hbox{$\nthinsp *$}}}}
\def\lamstarwscr
{\lambda\subwscr^{{\raise 0.5pt\hbox{$\nthinsp *$}}}}

\def\shortmtx#1#2{\renewcommand{\arraystretch}{0.9}%
      \left( \begin{array}{#1}#2\end{array} \right)}

\def\lamactT{{\vphantom{\rule{1pt}{5.5pt}}\smash{\overset{\;\,=}{\vphantom{a}\smash{\lambda}}}}^T}

\def\Yscr{{\cal Y}}

\def\lamstarwscrT{\lambda\subwscr^{{\raise 0.5pt\hbox{$\nthinsp *$}}T}}

\def\Iscr{{\cal I}}
\def\Escr{{\cal E}}
\def\Aeq{A\subescr}
\def\beq{b\subescr}
\def\meq{m\subescr}
\def\Aineq{A\subiscr}
\def\bineq{b\subiscr}
\def\mineq{m\subiscr}

\def\subescr{_{\scriptstyle\Escr}}

\def\subiscr{_{\scriptstyle\Iscr}}
\def\subiscrback{_{\!\scriptstyle\Iscr}}

\def\Iscrinact{\overset{{\scriptscriptstyle\;\,>}}{\vphantom{b}\smash{\Iscr}}}
\def\Iscract{\overset{\;\,=}{{\rule{0pt}{1.2ex}}\smash{\Iscr}}}

\def\lamescrT{\lambda_{\scriptstyle{\Escr}}^T}
\def\lamiscrT{\lambda_{\scriptstyle{\Iscr}}^T}

\def\lamstariscr{\lambda\subiscr^{{\raise 0.5pt\hbox{$\nthinsp *$}}}}
\def\lamstariscrT{\lambda\subiscr^{{\raise 0.5pt\hbox{$\nthinsp *$}}T}}

\def\lamstarescr{\lambda\subescr^{{\raise 0.5pt\hbox{$\nthinsp *$}}}}
\def\lamstarescrT{\lambda\subescr^{{\raise 0.5pt\hbox{$\nthinsp *$}}T}}

\begin{document}

\title{An elementary proof of\\
linear programming optimality conditions\\
without using Farkas' lemma
}

\author{
           Anders Forsgren\\
           Optimization and Systems Theory\\
           Department of Mathematics\\
           KTH Royal Institute of Technology\\
           Stockholm, Sweden\thanks{andersf@kth.se}
\and
           Margaret H.\ Wright\\
           Department of Computer Science\\
           Courant Institute of Mathematical Sciences\\
           New York University\\
           New York, New York, USA\thanks{mhw@cs.nyu.edu}
}

\maketitle

\begin{abstract}
Although it is easy to prove
the sufficient conditions for optimality of a linear
program, the necessary conditions pose a
pedagogical challenge.  A widespread
practice in deriving the necessary conditions
is to invoke Farkas' lemma, but proofs of
Farkas' lemma typically involve ``nonlinear'' topics
such as separating
hyperplanes between disjoint convex sets, or else more
advanced LP-related material such as duality and anti-cycling
strategies in the simplex method. An alternative
approach taken previously by several authors is to
avoid Farkas' lemma through a direct proof of the necessary
conditions.  In that spirit,
this paper presents what we believe to be an ``elementary''
proof of the necessary conditions that does not rely on
Farkas' lemma and is independent of
the simplex method, relying only on
linear algebra and a perturbation
technique published in 1952 by Charnes.  No claim
is made that the results are
new, but we hope that the proofs may be useful for
those who teach linear programming.
\end{abstract}

\section{Introduction}

In many contexts, particularly in business and economics,
linear programming (LP) is taught as a self-contained
subject, including proofs of necessary and sufficient
conditions for optimality.  The
proof of sufficient conditions is straightforward, but,
as explained below, the necessary conditions are seen by
many as pedagogically challenging.
Broadly speaking, these conditions
are proved in two different ways.

The first invokes 
Farkas' lemma,\footnote{The {\sl Chicago Manual of Style Online},
{\tt www.chicagomanualofstyle.org}, prefers repeating the ``s''
after the apostrophe in indicating possession by a word
ending in ``s'', but states that it is also correct to omit the
post-apostrophe ``s''.  We have chosen the second option.}
which can be stated in a
surprisingly large number of forms (for example,
\cite[pages 89--93]{Schrijver}) and which
itself must be proved.
According to the history sketched in \cite{Broyden},
the initial
statement of Farkas' lemma was published in 1894, with its
best-known
exposition appearing in 1902.  Despite the passage of more
than a century
since its correctness was established,
different proofs of the lemma have continued to
be devised, based on an array of motivations
categorized in \cite{Broyden} as geometric,
algebraic, and/or algorithmic.

A widely used means of proving Farkas' lemma relies on
separating hyperplane theorems (for example,
\cite[pages 205--207]{Fletcher},
\cite[pages 297--301]{GMW},\cite[pages 170--172]{BT}),
but there can be some discomfort in bringing these
more advanced topics into an LP course
whose audience is familiar only with basic linear algebra.
Even so, this approach is especially convenient when
optimality conditions for a range of
increasingly complicated
constrained optimization problems are to be presented
(for example, \cite[pages 326--329]{NocWright}).

Those who wish to keep Farkas' lemma but
prefer to avoid separating hyperplanes can choose instead from
a non-trivial number of ``elementary''
proofs of the lemma.  Some of these, such as \cite{Dax, Svanberg},
involve properties of linear least-squares problems.
An algebraic proof related to orthogonal
matrices is given in \cite{Broyden} (see also \cite{RoosTerlaky}),
and \cite{Bartl07} features a linear-algebraic approach to
proving
Farkas' lemma and other theorems of the alternative.

A second strategy is to prove the necessary conditions
for LP optimality without explicitly calling on Farkas' lemma.
This can be done in a variety of ways, for example using
LP-related results such as
duality (see \cite[page 165]{BT},
\cite[pages 112-113]{FMW}) or finite termination of
the simplex method with a guaranteed anticycling strategy
\cite[page 86]{Schrijver}.  A recent proof of
LP optimality conditions \cite{Forsgren} 
is independent of the simplex method, relying instead on
linear algebra and
a perturbation technique introduced by \cite{Charnes} in
the context of resolving degeneracy.

This paper, which falls into the
second group, presents proofs of optimality
conditions for linear programs expressed in
a generic form that includes both equalities and inequalities;
see (\ref{eqn-togetherform}).
The problem form may seem inconsequential, especially
since all known LP problem forms can be mechanically
transformed into one another.  But form
affects substance to a perhaps surprising extent,
and can have
a major effect on how students and practitioners think about 
linear programs and algorithms for solving them.  

Linear programs are (probably) most frequently expressed in
textbooks using one of several variations
on ``standard form'', of which the following is typical:
\begin{equation}
\label{eqn-stdform}
\minimize{x\in\Re^n}\;\; c\T x 
\quad\hbox{subject to}\quad Ax = b
\quad\hbox{and}\quad x\ge 0, 
\end{equation}
where $A$ is $m\times n$ with $m\le n$, $b\in\Re^m$,
$c\in\Re^n$, and $A$ is assumed
to have rank $m$. 
Two key features of this version of standard form
are that the ``general'' constraints involving
$A$ are all equalities, and that the only inequalities are
simple lower bounds on the variables.

Standard form is very closely
tied to the simplex method, which is described in many
papers and books (for example, the 1963
classic \cite{Dantzig}, \cite{Chvatal} and 
\cite{Vanderbei}) and which was, for almost 40 years,
essentially the only method for solving linear programs.
However, since the 1984 ``interior-point revolution'' in
optimization (for example, \cite{MHWint,NemTodd}),
a thorough treatment of linear programming requires
presentation of interior-point methods.
These are easier to
motivate with {\sl all-inequality form}, which resembles
a generic form for constrained optimization: 
\begin{equation}
\label{eqn-allineqform}
\minimize{x\in\Re^n}\;\; c\T x 
\quad\hbox{subject to}\;\; Ax \ge b,
\quad\hbox{where $A$ is $m\times n$}. 
\end{equation}

A linear program in standard form (\ref{eqn-stdform}) may be
transformed (by reformulating the constraints and/or adding
variables) into an equivalent linear program in all-inequality form
(\ref{eqn-allineqform}) and the other way around. This paper considers
a generic mixed form in which equality and inequality constraints are
denoted separately:
\begin{equation}
\label{eqn-togetherform}
\minimize{x\in\Re^n}\;\; c\T x 
\quad\hbox{subject to}\;\; \Aeq x  = \beq
\;\;\hbox{and}\;\; \Aineq x \ge \bineq, 
\end{equation}
where $\Aeq$ is $\meq\times n$ with
$\rank(\Aeq) = \meq$ and $\Aineq$ is $\mineq\times n$.
This form corresponds to all-inequality
form when $\meq = 0$, i.e., when $\Aeq$ is empty,
and to standard form when $\Aineq = I_n$ (the
$n$-dimensional identity) and $\bineq=0$.
This means that results on standard form as well as all-inequality
form are immediately available from our results.
We define the combined matrix $A$ and vector $b$, each with
$m = \meq + \mineq$ rows, as
\begin{equation}
\label{eqn-abfulldef}
  A = \mtx{c}{\Aeq\\
              \Aineq}
\quad\hbox{and}\quad
  b = \mtx{c}{\beq\\
              \bineq}.
\end{equation}
where the index sets ${\Escr}$ and ${\Iscr}$ are
${\Escr} = \{1,2,\dots,\meq\}$ and
${\Iscr} = \{\meq + 1, \dots, \meq + \mineq\}$.

Sections 2--4 contain a summary of background results
that would be part of any course on linear programming;
they are included
to make the paper self-contained.
The results in Sections 5--8 are not new in substance, but
may be unfamiliar in form.  In any case
we hope that they might provide a useful option for
proving LP optimality.
For completeness, Section~\ref{sec-farkas}
states and proves Farkas' lemma using
the results in this paper; Section~\ref{sec-summary} summarizes
the logical flow of results.

\section{Notation, definitions, and background results}

It is assumed that $c\ne 0$ and $A \ne 0$.
The $i$th row of $A$ (\ref{eqn-abfulldef}) is denoted
by $a_i^T$ and the $i$th component of $b$ by $b_i$. 
The problem constraints are said to be
{\sl consistent\/} or {\sl feasible\/} if there exists at least
one $\xhat$ such that $\Aeq \xhat = \beq$
and $\Aineq\xhat \ge \bineq$,
and an $\xhat$ that satisfies the constraints
is called a {\sl feasible point}.
An immediate result, noted explicitly for completeness, is
that linearity of the constraints means that
every point on the line joining two distinct feasible
points $\xhat$ and $\xbar$ is also feasible.

Optimality subject to constraints is inherently a relative
condition, involving
comparison of objective function values at 
a possible optimal point $\xhat$ with those at other feasible
points.
In such a comparison, {\sl active\/} constraints play a crucial role.
The $i$th constraint is said to be {\sl active\/} at a feasible point $\xhat$
if $a_i^T \xhat = b_i$.
At a feasible point,
the equality constraints $\Aeq\xhat = \beq$ must be active,
but an inequality constraint may be active
or inactive (strictly satisfied).
The set of indices of inequality constraints active at
a feasible point $\xhat$ is denoted by $\Iscract(\xhat)$,
and $\Iscrinact(\xhat)$ denotes the
set of indices of the inactive inequality constraints
at $\xhat$.
We use $\Ascr(\xhat)$ to denote the set of active
constraints, which means that $\Ascr(\xhat) = \Escr \cup \Iscract(\xhat)$.
Let $\Aact\subiscrback(\xhat)$ denote the matrix
of rows of $\Aineq$ corresponding to
active inequality constraints, and similarly for
$\bact\subiscr(\xhat)$, so that, by definition,
$\Aact\subiscrback(\xhat) \xhat = \bact\subiscr(\xhat)$.
The {\sl active-constraint matrix\/} $\Aact(\xhat)$ then consists of
$\Aeq$ and $\Aact\subiscrback(\xhat)$:
\begin{equation}
\label{eqn-actmatdef}
  \Aact(\xhat) = \mtx{c}{\Aeq\\
                          \Aact\subiscrback(\xhat)}.
\end{equation}

The following definition, in which positivity of $\alpha^i$ is
crucial, allows us to characterize feasible directions at a feasible point.
\begin{definition}[Feasible direction.]
\label{def-feasdir}
The $n$-vector  
$p$ is a {\sl feasible direction\/} for the
constraints $\Aeq x = \beq$ and $\Aineq x \ge \bineq$
at the feasible
point $\xhat$ if $p\ne 0$ and there exists $\alpha^i > 0$
such that $\xhat + \alpha p$ is feasible for
$0 < \alpha\le \alpha^i$, where $\alpha^i$ may be infinite.
\end{definition}

By linearity, the value of the
$i$th constraint when moving from a feasible point $\xhat$ to
$\xhat + \alpha p$, where $p\ne 0$ and $\alpha > 0$, is given by
\begin{equation}
\label{eqn-xhatpert}
  a_i^T (\xhat + \alpha p) = a_i^T \xhat + \alpha a_i^T p.
\end{equation}
Relation (\ref{eqn-xhatpert}) shows that, to maintain
feasibility with respect to an equality constraint $i$, $p$
must
satisfy $a_i^T p = 0$.
When $i$ is an {\sl inactive\/} inequality constraint,
(\ref{eqn-xhatpert}) implies
that, even if $a_i^T p < 0$, constraint $i$
will remain satisfied at $\xhat + \alpha p$ if $\alpha> 0$ is
sufficiently small.
But if $i$ is an active inequality constraint, so that $a_i^T \xhat = b_i$, it follows from
(\ref{eqn-xhatpert}) that $\xhat + \alpha p$
will be feasible with respect to constraint $i$
for $\alpha > 0$ only if $a_i^T p \ge 0$.

Although inactive inequality constraints do not have an immediate
local effect on feasibility, they may limit the size
of the step that can be taken along a feasible direction $p$.  
\begin{definition}[The maximum feasible step.]
\label{def-maxfeasible}
Given a feasible point $\xhat$ and a feasible direction $p$, let
$\Dscr(\xhat, p)$ (for ``decreasing'') denote the set of
indices of inequality constraints that are inactive at $\xhat$
for which $a_i^T p < 0$:
$$
  \Dscr(\xhat,p) \definedas 
  \{ i \mid i\in\Iscrinact(\xhat)\;\;\hbox{and}\;\;
  a_i^T p < 0\}.
$$
If $\Dscr(\xhat,p) \ne \emptyset$, 
the positive scalar $\sigma^i$ (the
step to constraint $i$ along $p$) is defined as
\begin{equation}
\label{eqn-steptoi}
  \sigma^i \definedas \frac{b_i-a_i^T \xhat}{a_i^T p}
\quad\hbox{for $i\in\Iscrinact(\xhat)$ and $a_i^T p < 0$.}
\end{equation}
The {\sl maximum feasible step}, denoted by $\sigmahat$,
is the smallest such step,
$\sigmahat \definedas\min_{i\in\Dscr(\xhat,p)} \sigma^i$.
Any inequality constraint $i$ for which $\sigma^i = \sigmahat$,
of which there may be more than one,
becomes active at $\xhat + \sigmahat p$.  If
$\Dscr(\xhat,p) = \emptyset$, $\sigmahat$ is taken as $+\infty$.
\end{definition}

In addition to feasibility, optimality
conditions need to ensure that the objective
function is as small as possible.

\begin{definition}[Descent direction.]
The vector $p$ is a {\sl descent direction} for the objective
function $c^T x$ if $c^T p < 0$.
\end{definition}

\begin{definition}[Feasible descent direction.]\label{def-feasdesc}
The direction $p$ is a feasible descent direction
at a feasible point $\xhat$ if $\Aeq p = 0$,
$\Aact\subiscrback(\xhat) p \ge 0$, and
$c^T p < 0$.  No feasible descent direction exists at
$\xhat$ if there are no feasible directions or if
$c^T p \ge 0$ for all feasible directions $p$.
\end{definition}

Obvious optimality conditions can now be stated in terms of
existence or non-existence of a feasible descent direction.
\begin{lemma}[Necessary and sufficient optimality
conditions---Version I.]
\label{lem-necsuff-one}
\nextline
When minimizing $c^T x$ subject to
$\Aeq x = \beq$ and $\Aineq x\ge \bineq$, the feasible point $\xstar$ is optimal
if and only if no feasible descent direction
exists at $\xstar$.
\end{lemma}
\begin{proof}
The ``only if'' result follows because existence of
a feasible descent direction $p$ at a feasible point
$\xstar$ implies that there is a positive $\alpha$ such that
$\xstar + \alpha p$ is feasible and
$c^T (\xstar + \alpha p) = c^T\xstar + \alpha c^T p < c^T
\xstar$. Hence, $\xstar$ cannot be optimal.

To show the ``if'' result, assume that $\xstar$ is feasible but not
optimal. Then there is a feasible point $\xtilde$ such that $c^T
\xtilde < c^T \xstar$. Since $\xtilde$ is feasible, we must have $\Aeq
\xtilde = \beq$ and $\Aact\subiscrback(\xstar) \xtilde \ge
\bact\subiscrback(\xstar)$. In addition, it holds that $\Aeq \xstar =
\beq$ and $\Aact\subiscrback(\xstar) \xstar =
\bact\subiscrback(\xstar)$. Hence, $c^T (\xtilde-\xstar)<0$, $\Aeq
(\xtilde -\xstar) = 0$ and $\Aact\subiscrback(\xstar) (\xtilde
-\xstar) \ge 0$, so that $\xtilde-\xstar$ is a feasible descent
direction by Definition~\ref{def-feasdir}.
\end{proof}

Although Lemma~\ref{lem-necsuff-one} gives necessary
and sufficient conditions for LP optimality,
its usefulness is limited because it offers no way
to verify these conditions.   This is the point in teaching LP
where Farkas' lemma usually enters the picture, but
we now take a different route
to the necessary and sufficient conditions for LP optimality.

\section{Multipliers and optimality}

An important feature of constrained optimization problems is
the implicit existence of quantities
that do not appear in the problem statement
yet play a crucial role in optimality conditions. These
quantities consist of $m$ (Lagrange) {\sl multipliers},
or {\sl dual variables}, one for each
constraint, that connect the objective
and the constraints.  The next result shows that
existence of a multiplier
with certain properties produces a lower bound on the
objective value in the feasible region.

\begin{proposition}[Lower bound on LP objective.]
\label{prop-lowerbound}
Assume that $x\in\Re^n$ satisfies $\Aeq x = \beq$ and $\Aineq x\ge
\bineq$.  Further assume that there exists a multiplier
$\lambda\in\Re^m$ such that $A^T\lambda = c$ and $\lambda\subiscr\ge
0$, where $\lambda\subiscr$ denotes the $\mineq$-vector of components
of $\lambda$ corresponding to inequality constraints.  (No sign
restrictions apply to the multipliers for equality constraints.)
Then $c^T x - \lambda^T b = \lambda^T (Ax-b) \ge 0$.
\end{proposition}
\begin{proof}
  Let $x$ be feasible. The assumed existence of $\lambda$ means that
  we can substitute $A^T\lambda$ for $c$ and use the facts that
  $\lambda\subiscr \ge 0$, $\Aeq x - \beq = 0$, and and
  $\Aineq\xhat-\bineq\ge 0$.  We then have
$$ 
c^T x - \lambda^T b =  
\lambda^T (A x -  b) = \lamescrT(\Aeq x - \beq)
+ \lamiscrT(\Aineq x - \bineq) \ge 0.
$$ 
It follows that $c^T x$ is bounded below by
$\lambda^T b$ for every feasible $x$.
\end{proof}

For a general linear program, a qualifying
$\lambda$ may not exist.  Furthermore, there can
be more than one vector $\lambda$ satisfying the given
conditions, each producing a different value of $\lambda^T b$.  
However, something special happens if
$\lamiscrT(\Aineq\xhat-\bineq) = 0$, allowing us to state
{\sl sufficient\/} conditions for LP optimality.

\begin{proposition}[Sufficient conditions for LP optimality.]
\label{prop-suffone}
Consider the linear program of minimizing
$c^T x$ subject to the consistent constraints
$\Aeq x = \beq$ and $\Aineq x\ge \bineq$.  The feasible point $\xstar$ is optimal
if a multiplier
$\lamstar\in\Re^m$ exists with the following three properties:
(i) $A^T\lamstar = c$,
(ii) $\lamstariscr\ge 0$, and (iii) $\lamstarT(A\xstar-b) = 0$.
The optimal objective value is $c^T \xstar = \lamstarT b$.
\end{proposition}
\begin{proof}
Since $\lamstar$ satisfies (i) and (ii),
Proposition~\ref{prop-lowerbound} gives the lower bound $\lamstarT b$
on the optimal value of the linear program. However, (iii) and
Proposition~\ref{prop-lowerbound} show that the lower bound is
attained for $\xstar$, so that $\xstar$ is optimal.
\end{proof}

The crucial relationship $\lamstarT(A\xstar-b) = 0$, which means
that, for every $i=1,\dots, m$, at least one of
$\{\lamstari, a_i^T\xstar - b_i\}$ must be zero,
is called {\sl complementarity}.

The following result shows that the complementarity condition does not
directly tie the multiplier $\lamstar$ to a particular $\xstar$.
Rather, if complementarity holds for one $\xstar$, it must hold for $\lamstar$
together with any optimal solution $\xhat$.

\begin{proposition}[Properties of an optimal LP solution.]
\label{prop-specialprops}
Consider minimizing $c^T x$ subject to the consistent constraints
$\Aeq x = \beq$ and $\Aineq x\ge \bineq$.  Assume that a multiplier
$\lamstar$ exists such that $A^T\lamstar = c$,
$\lamstariscr \ge 0$,
and assume that the optimal objective value is $\lamstarT b$.
Then a
feasible point $\xhat$ is optimal if and only if
$\lamstarT(A\xhat -b) = 0$.
\end{proposition}
\begin{proof}
For any feasible $x$, Proposition~\ref{prop-lowerbound} gives $c^T x -
\lamstarT b = \lamstarT (A x - b) \ge 0$. Hence, under the assumption
that the optimal value is $\lamstarT b$, a point $\xhat$ is optimal if and
only if $\lamstarT(A\xhat - b) = 0$.
\end{proof}

\section{Vertices and their properties}

Certain feasible points, known as {\sl vertices}, are extremely
important in linear programming.
\begin{definition}[Vertex.]
\label{def-vertex}
Given the consistent constraints $\Aeq x = \beq$
and $\Aineq x \ge \bineq$, the point $\xhat$ is a
vertex if $\Aeq \xhat = \beq$, $\Aineq \xhat \ge \bineq$, and
the active-constraint matrix $\Aact(\xhat)$ of (\ref{eqn-actmatdef})
has rank $n$.
\end{definition}
An immediate consequence is that
there cannot be a vertex when the rank of the full constraint
matrix $A$ (\ref{eqn-abfulldef}) is less than $n$.
A nice feature of standard-form linear programs
(\ref{eqn-stdform}) is that, when the constraints are
consistent, a
vertex must exist because the inequality constraints
consist of the $n$-dimensional identity.
This is not true in general for all-inequality form
(\ref{eqn-allineqform}),
even when the objective function is bounded below in the
feasible region;
consider, for example, minimizing $x_1 + x_2$ subject to
$x_1 + x_2 \ge 1$.

A vertex $\xhat$ is the unique solution of the linear
system formed by any nonsingular submatrix of $\Aact(\xhat)$,
which has rank $n$ by definition.
A fundamental result is that the definitions of vertex
and extreme point are equivalent, where
an extreme point
is a feasible point that does not lie on the line segment
joining two distinct
feasible points. (See, for example, Section 2.2 of
\cite{BT} for a detailed treatment of related topics.)

A simple combinatorial argument shows that
the number of vertices is bounded above
by $\tbinom{m}{n}$.
Given the consistent constraints
$\Aeq x = \beq$ and $\Aineq x \ge \bineq$, the
set of vertices $\Vscr(A,b)$ can be found by enumerating
and testing for feasibility
every combination of the
constraints in which the equality constraints
hold with equality and a total of $n$ constraints hold with equality.
(Of course, this procedure is not practical for large $m$ and $n$.)

There are two kinds of vertices.
At a {\sl nondegenerate vertex\/} $\xhat$,
exactly $n$ constraints are active and
the active-constraint matrix $\Aact(\xhat)$ is nonsingular.
At a {\sl degenerate vertex\/}
$\xhat$, there are $n$ linearly independent active constraints,
but more than $n$ constraints are active.

The next two small results are stated formally for
later reference.
\begin{result}\label{res-inactindep}
Let $F$ be a $q\times n$ nonzero matrix with $\rank(F) < n$
whose $i$th row is $f_i^T$.  Assume that $p$ is a nonzero
$n$-vector such that $Fp=0$.
If $g$ is a vector such that $g^T p \ne 0$, then $g^T$ is
linearly independent
of the rows of $F$, i.e.
$$
  \rank\shortmtx{c}{F\\
               g^T} = \rank(F) + 1.
$$
\end{result}
\begin{proof}
If $g^T$ were a linear combination of the rows of $F$,
then $g^T = y^T F$ for some vector $y$.  Since
$Fp=0$, substituting $y^T F$
for $g^T$ would give $g^T p = y^T F p = 0$,
contradicting our assumption that $g^T p \ne 0$.
\end{proof}

Note that the implication in Result~\ref{res-inactindep}
does not go the other way: if $Fp = 0$, $p\ne 0$, and
$g^T p = 0$, then $g^T$ can nonetheless be linearly
independent of
the rows of $F$.  This can be seen by example:
$$
  F = \mtx{rrrr}{1 & 1 & 0 & 0\\
                 0 & 1 & 0 & -1},
\quad
g^T = \mtx{cccc}{0 & 0 & 1 & 0},
\quad\hbox{and}\quad
p^T = \mtx{cccc}{-1 & 1 & 0 & 1}.
$$

\begin{result}\label{res-mustbeonea}
Let $D$ be an $m\times n$ matrix with $\rank(D) = n$
whose $i$th row is $d_i^T$.  Let $\Dtilde$ denote a subset of
rows of $D$ such that $\rank(\Dtilde) = r < n$, and
assume that $p$ is a nonzero vector such that
$\Dtilde p = 0$.  Then there is at least one row $d_j^T$ of $D$
that is not included in $\Dtilde$ such that (i) $d_j^T p \ne 0$
and (ii) $d_j^T$ is linearly independent of the rows of
$\Dtilde$.
\end{result}
\begin{proof}
Let the $r\times n$ matrix $F$ consist of $r$ linearly
independent rows of $\Dtilde$, so that every row
in $\Dtilde$ that is not in $F$ is a non-trivial linear
combination of the rows of $F$.  Hence the assumption that
$\Dtilde p = 0$ implies that $Fp=0$.

Because $\rank(\Dtilde) = r$ and
$\rank(D) = n$, we can assemble a matrix $G$
consisting of $n-r$ rows of $D$ 
that are not in $\Dtilde$ and that are linearly
independent of the rows of $\Dtilde$, such that
the $n\times n$ matrix 
$$
  M = \shortmtx{c}{F\\
                   G}
\;\;\hbox{is nonsingular}.
$$
Since $p\ne 0$, nonsingularity of
$M$ means that $Mp\ne 0$ and, since $Fp=0$, this will be
true only if
$Gp\ne 0$.  Given how $G$ is defined,
there must be a row $d_j^T$ in $D$
but not in $\Dtilde$ such that $d_j^T p \ne 0$,
and linear independence of $d_j^T$ follows
directly from Result~\ref{res-inactindep}.
\end{proof}

Our next step is to determine when there is an
{\sl optimal\/} vertex for the LP (\ref{eqn-togetherform}).
We know that a vertex can
exist only if the constraints are consistent and $\rank(A) = n$.
Using a theoretical procedure,
the next lemma guarantees the existence of an {\sl optimal\/}
vertex under the added assumption
that $c^T x$ is bounded below in the feasible region.

\begin{lemma}[Existence of an optimal vertex.]\label{lem-optvert} 
Consider minimizing $c^T x$ subject to the consistent
constraints $\Aeq x = \beq$ and $\Aineq x\ge \bineq$, where
the rank of $A$ (\ref{eqn-abfulldef}) is $n$.
Let $\Vscr$ denote the set of all
vertices for the given constraints. Then either
\begin{enumerate}
\item[(i)] $c^T x$ is bounded below in the feasible region and there
is a vertex $\vstar\in\Vscr$ where the smallest value of $c^T x$ in
the feasible region is achieved; or
\item[(ii)] $c^T x$ is unbounded below in the feasible region and there
exists an $n$-vector $p$ such that $\Aeq p = 0$,
$\Aineq p\ge 0$, and $c^T p < 0$. 
\end{enumerate}
\end{lemma}
\begin{proof}
Starting with any feasible point $x_0$, we define
an iterative sequence $\{x_k\}$ that produces
a vertex $v_j\in\Vscr$ such that $c^T v_j \le c^T x_0$, unless we find
an $n$-vector $p$ such that $\Aeq p = 0$,
$\Aineq p\ge 0$, and $c^T p < 0$. 
At $x_k$, $\Aact_k$ denotes the
active-constraint matrix $\Aact(x_k)$
defined by (\ref{eqn-actmatdef}); note that the
rows of $\Aeq$ are always present in $\Aact_k$.
\begin{description}
\item[Step 0.]
Set $k=0$. 
\item[Step 1.]
If $x_k$ is a vertex (i.e., $\rank(\Aact_k) = n$) 
then $x_k = v_j\in\Vscr$ for some $j$.
Stop; a vertex has been found such that
$c^T v_j \le c^T x_0$.  Otherwise, go to Step 2.

\item[Step 2.] Since $\rank(\Aact_k) < n$,
there exists a nonzero $p$ satisfying $\Aact_k p = 0$, so that
any movement along $p$ does not alter the values
of constraints active at $x_k$.
Now we consider the inequality constraints that are
inactive at $x_k$.
It follows from Result~\ref{res-mustbeonea} with
$\Aact_k$ playing the role of $\Dtilde$ that there must be
at least one inactive inequality constraint index $j$ such that
$a_j^T p \ne 0$; let $\Jscr_k = \{j\mid a_j^T x_k > b_k
\;\;\hbox{and}\;\; a_j^T p \ne 0\}$.

\item[Step 3.]
If $c^T p = 0$, we select $j\in\Jscr_k$
and set $p_k = \pm p$, choosing the sign
so that $a_j^T p_k < 0$ (since either choice
will satisfy $\Aact_k p_k = 0$).
Otherwise, if $c^T p \ne 0$, choose
$p_k = \pm p$ so that, for some $j\in\Jscr_k$, 
$a_j^T p_k < 0$ and $c^T p_k < 0$; if this is not possible
it must hold that $c^T p_k < 0$
and $a_j^T p_k \ge 0$ for all $j\in\Jscr_k$, in which case
we exit and conclude that (ii) holds.

Applying Definition~\ref{def-maxfeasible},
let $\alpha_k>0$ be the maximum feasible step along $p_k$.
Then
all the constraints inactive at $x_k$ remain feasible
at  $x_{k+1} = x_k + \alpha_k p_k$, and at least one
additional linearly independent inequality constraint becomes active
there.  Hence $\rank(\Aact_{k+1}) > \rank(\Aact_k)$
and $c^T x_{k+1} \le c^T x_k$.

\item[Step 4.] 
Increase $k$ to $k+1$ and return to Step 1.

\end{description}

For each initial $x_0$ there will be no more than $n$
executions of Step 1, since $\rank(A) = n$
and each pass through Step 3 increases the
rank of $\Aact_k$ by at least one.

This procedure confirms that, for every feasible
point $x_0$, there is either a vertex $v_j\in\Vscr$ such
that $c^T v_j \le c^T x_0$ or an $n$-vector $p$ such that $\Aeq p = 0$,
$\Aineq p\ge 0$, and $c^T p < 0$. 
Let $\vstar$ denote a vertex
such that $c^T \vstar \le c^T v_j$ for all
$v_j$ in the finite set $\Vscr$. If there is a feasible $x_0$ such that $c^T
x_0 < c^T\vstar$, the procedure must give a $p$ such that $\Aeq p = 0$,
$\Aineq p\ge 0$, and $c^T p < 0$ and (ii) holds. Otherwise, $c^T
x_0\ge c^T \vstar$ for all feasible $x_0$ and (i) holds.
\end{proof}

\section{Optimality at a nondegenerate vertex}

It is straightforward to derive
necessary and sufficient conditions for optimality
of a {\sl nondegenerate\/} vertex.

\begin{proposition}
[Optimality of a nondegenerate vertex.]
\label{prop-nondegenopt}
Consider the linear program of minimizing $c^T x$
subject to the consistent constraints $\Aeq x = \beq$
and $\Aineq x\ge \bineq$.
Assume that $\xstar$ is a nondegenerate
vertex where the active set is $\Ascr(\xstar)$ and that
the $n$-vector $\lamact$ is the solution of
$\Aact(\xstar)^T \lamact = c$.  Then
$\xstar$ is optimal if and only if $\lamact\subiscr\ge 0$,
where $\lamact\subiscr$ denotes the components of
$\lamact$ corresponding to active inequality constraints.
\end{proposition}
\begin{proof}
Because $\xstar$ is nondegenerate, $\Aact(\xstar)$ is nonsingular,
which means that $\lamact$ is unique.
The ``if'' direction follows
because, as we show next, we can define an $m$-vector $\lamstar$
that satisfies the sufficient conditions of
Proposition~\ref{prop-suffone}.
Assume that the rows of $\Aact(\xstar)$ are ordered with indices
$\{w_1, \dots, w_n\}$,
so that the $j$th component of $\lamact$
is the multiplier for original
constraint $w_j$. The full multiplier $\lamstar$ is then
defined as
\begin{equation}
\label{eqn-lamstardef}
 \lamstar_{w_j}  =  \lamact_j,\quad
 j = 1,\dots, n; \quad\hbox{and}\quad
 \lamstar_j  =  0 \quad\hbox{if $j\notin\Ascr(\xstar)$,}
\end{equation}
so that the multipliers corresponding to inactive inequality
constraints are zero.
Hence $\lamstarT(A\xstar -b) = 0$,
the sufficient conditions of
Proposition~\ref{prop-suffone}
are satisfied, and $\xstar$ is optimal.

For the ``only if'' direction,
suppose that $[\lamact\subiscr]_i$ is strictly
negative for some active inequality constraint.
Because $\Aact(\xstar)$ is nonsingular, there is a
unique direction $p$ satisfying
$\Aact(\xstar) p = e_i$, where $e_i$ is the $i$th
coordinate vector, so that $p$ is a feasible
direction.  It then follows from the
relation $c = \Aact(\xstar)^T \lamact$ that
$c^T p = \lamactT \Aact(\xstar) p = [\lamact\subiscr]_i < 0$ and
$p$ is a feasible descent direction, which means
that $\xstar$ cannot be optimal.
\end{proof}

\section{Optimality at a degenerate vertex}

Difficulties arise in proving necessary optimality conditions
for a degenerate optimal vertex
because Proposition~\ref{prop-nondegenopt}
depends on nonsingularity
of the active-constraint matrix at the vertex $\xstar$.
To address these difficulties, we use properties of a
{\sl working set}, which is closely related to, but not
the same as, the active set; see \cite[page 339]{GMW}
for a more restricted definition.

\begin{definition}[Working set.]
\label{def-working}
Given the consistent constraints $\Aeq x = \beq$ and $\Aineq x\ge \bineq$, let
$\xbar$ be a feasible point, not necessarily a vertex.
Consider a set of $n\subwscr$ distinct indices, 
$\Wscr = \{w_1,\dots, w_{n\subwscr}\}$, where $\meq \le n\subwscr\le n$ and
$w_i = i$ for $i = 1$, \dots $\meq$, i.e., the first $\meq$ indices
in $\Wscr$ are the indices of the equality constraints. 
Let $\Working$ be the associated $n\subwscr\times n$ working matrix
whose $i$th row is $a_{w_i}^T$, so that the first $\meq$ rows of $\Working$
are $\Aeq$.  Let $b\subwscr$ denote the vector consisting
of components of $b$ corresponding to the indices
in $\Wscr$.  Then $\Wscr$ is a {\emph working set\/} at $\xbar$
if the following two properties apply:
\begin{enumerate}
\item[(1)] Every inequality constraint whose index is in
$\Wscr$ is active at $\xbar$, i.e., $W\xbar = b\subwscr$;
\item[(2)] The rows of $\Working$ are linearly independent. 
\end{enumerate}
\end{definition}

At a vertex, by definition
there are $n$ linearly independent active constraints, so that it is always
possible to define a nonsingular working-set matrix $\Working$ and
an associated unique vector $\lambda\subwscr$ that satisfies
$\Working^T \lambda\subwscr = c$, where component
$j$ of $\lambda\subwscr$ corresponds to original constraint $w_j$.
Thus components $1$, \dots, $\meq$ of $\lambda\subwscr$ correspond
to the $\meq$ equality constraints, and the remaining 
components are associated with ``working'' (active)
inequality constraints.

If $\xstar$ is an optimal nondegenerate vertex, the active
set $\Aact(\xstar)$ is a working set with $n\subwscr = n$, $\lambda\subwscr$ is
the same as the unique solution $\lamact$ of
$\Aact(\xstar)^T \lamact = c$,
and $[\lambda\subwscr]_i \ge 0$ for $w_i\in\Iscract(\xstar)$.
But if $\xhat$ is an optimal vertex that is degenerate,
there can be more than one working set.  This complicates
optimality conditions because, even if $\widehat\Wscr$ is a working
set at a degenerate optimal vertex $\xhat$, it may not be true that
$[\lambda\subwscrhat]_i \ge 0$ for inequality constraints
in the working set.

Consider, for example, the following all-inequality
two-variable
linear program of  minimizing $c^T x$ subject to three
inequality constraints $Ax\ge b$, with
\begin{equation}
\label{eqn-pertdegeninfig}
c = \thinmtx{r}{1\,\\
              -\half},\quad
A = \mtx{cr}{1 & 1\\
             1 & \frac{5}{2}\\
             1 & -2},\quad\hbox{and}\quad
b =\thinmtx{r}{3\\
           6\\
           -3}.
\end{equation}
(Note that, for this example, there are no equality constraints.)
The optimal solution is a degenerate vertex $\xstar = (1,2)^T$.
There are three working sets,
$\Wscr_1 = \{1, 2\}$, $\Wscr_2 = \{1, 3\}$, and
$\Wscr_3 = \{2, 3\}$, and the associated multipliers are
$$
  \lambda\subwscrone = \mtx{r}{2\\
                            -1}, \quad
  \lambda\subwscrtwo = \mtx{c}{\half\\[3pt]
                            \half},
  \quad\hbox{and}\quad
   \lambda\subwscrthree = \mtx{c}{\third\\[3pt]
                               \twothirds}.
$$
Although $\Wscr_1$ identifies two linearly independent rows of
$A$, optimality of $\xstar$ cannot be determined by
checking the signs of the components of $\lambda\subwscrone$.

We therefore define an {\sl optimal working set\/} at an
optimal point as one that will confirm optimality, as
happens with working sets $\Wscr_2$ and $\Wscr_3$ in
the example.
\begin{definition}[Optimal working set.]
\label{def-optworking}
Given the consistent constraints $\Aeq x = \beq$ and
$\Aineq x\ge \bineq$ and the objective
function $c^T x$, assume that $\Wscr$ is a working set at the
optimal point $\xbar$ (not necessarily a vertex), with
$\Working$ the associated working matrix.  Then
$\Wscr$ is an \emph{optimal working set} if
(a) the linear system
$\Working^T \lambda\subwscr = c$ is compatible
(which means that $\lambda\subwscr$ exists and is unique), and
(b) $[\lambda\subwscr]_i\ge 0$ if $w_i$ is the index
of an active inequality constraint.
\end{definition}
Note that the uniqueness mentioned in
property (a) follows
because the columns of $\Working^T$ are linearly
independent (Definition~\ref{def-working}).

The next proposition shows that, if the constraints
are consistent, $\rank(A) = n$, and $c^T x$
is bounded below
in the feasible region, then an optimal vertex and an
associated optimal working set always exist, even if
the vertex is degenerate.  An optimal working
set is obtained from the solution of
a {\sl perturbed linear program\/} where,
as in \cite{Forsgren},
the perturbations reflect motivation introduced in
\cite{Charnes}; 
see, for example, \cite{Dantzig} and
\cite[pages 34--35]{Chvatal}.
The crucial property of the perturbations is that
their presence guarantees existence of an optimal
{\sl nondegenerate\/} vertex for the perturbed
problem.

\begin{proposition}[Existence of an optimal vertex,
multiplier, and working set.]
\label{prop-findoptworking}
\nextline
Consider minimizing $c^T x$ subject to the
consistent constraints $\Aeq x = \beq$ and $\Aineq x\ge \bineq$, where
$c^T x$ is bounded below in the feasible region and
$\rank(A) = n$.  Then there are an
optimal vertex $\xstar$ and an associated optimal working
set $\Wscr = \{w_1, \dots, w_n\}$
of $n$ indices, such that the corresponding working-set matrix $W$ is
nonsingular, from which an $m$-dimensional multiplier $\lamstar$ can
be constructed such that (i) $A^T \lamstar = c$, (ii)
$\lamstarT(A\xstar -b) = 0$, and (iii) $\lamstariscr\ge 0$.
\end{proposition}
\begin{proof}
The proof has two parts: analyzing a perturbed
linear program,
and then using the resulting nondegenerate optimal vertex
for the perturbed problem
to define a solution and optimal working set for
the original problem.

\smallskip
\noindent
{\bf Part 1: Solving a perturbed linear program.}
Consider the perturbed linear program:
\begin{equation}
\label{eqn-pertlp}
\minimize{x\in\Re^n}\quad c\T x 
\quad\hbox{subject to}\quad \Aeq x = \beq \;\;\hbox{and}\quad
    \Aineq x \ge \bineq - e,
\end{equation}
where $e = (\epsilon, \epsilon^2, \dots, \epsilon^{m_{\Iscr}})^T$
and $\epsilon > 0$ is arbitrary and ``sufficiently small''.
(Note that the equality constraints are not perturbed.)
Because the constraints of the original LP are consistent, so
are the constraints of the perturbed problem.  The objective
$c^T x$ and constraint matrix $A$ are the same
in both the original and perturbed problems. 

We know from Lemma~\ref{lem-optvert} that there must be an
optimal vertex, denoted by $x_{\epsilon}$,
for the perturbed problem.  Let
$\Aact_\epsilon$ denote the active-constraint matrix at
$x_{\epsilon}$ with respect to the perturbed constraints
of (\ref{eqn-pertlp}).  Without loss of generality,
since $x_{\epsilon}$ may be degenerate,
we write the active constraints at $x_{\epsilon}$ as
\begin{equation}
\label{eqn-active-epsilon}
 \Aact_{\epsilon} x_{\epsilon} =
 \bact_{\epsilon}  - \eacteps,
\quad\hbox{with}\quad
\Aact_{\epsilon} =
   \mtx{c}{\Working_{\epsilon}\\
          Y_\epsilon},
  \;\;
 \bact_{\epsilon} = \mtx{c}{b\subwscreps\\
                            b\subyscreps},
 \;\;\hbox{and}\;\;
\eacteps = \mtx{c}{e\subwscreps\\
                           e\subyscreps},
\end{equation}
where $\Working_{\epsilon}$ is $n\times n$ and
nonsingular.
Let $\Wscr_{\epsilon}$ denote the set of $n$ indices
$\{w_1, \dots, w_n\}$ of the original constraints
corresponding to the rows of $\Working_{\epsilon}$,
where the matrix $\Aeq$ corresponding to the equality
constraints occupies the
first $\meq$ rows of $\Working_{\epsilon}$.
The remaining $n-\meq$
rows of $\Working_{\epsilon}$ and the rows of
$Y_{\epsilon}$ contain normals of active inequality
constraints whose indices are not known in advance.
The first $\meq$ components of $e\subwscreps$
are zero (since the equalities are not perturbed), followed by
$n-\meq$ distinct powers of $\epsilon$:
\begin{equation}
\label{eqn-epspowers}
e\subwscreps = (0,\dots, 0, \epsilon^{w_1}, \epsilon^{w_2},\dots, 
\epsilon^{w_{n-m_{\scriptscriptstyle{\Escr}}}})^T,
\quad\hbox{with}\quad 1 \le w_i \le \mineq,\;\;
  i=1,\dots, n-\meq.
\end{equation}

We next show by contradiction 
that $x_{\epsilon}$ must be nondegenerate
for all sufficiently small $\epsilon$, i.e., that
$Y_{\epsilon}$ must be empty.
Let $y^T$ denote the normal of an inequality constraint
in $Y_{\epsilon}$, and assume that it corresponds to the $j$th
inequality constraint of the original problem.
Because $\Working_{\epsilon}$ is nonsingular, there
is a unique vector $q$ such that
$y^T = q^T {\Working}_{\epsilon}$.  Consequently,
since
$\Working_{\epsilon} x_{\epsilon} = b\subwscreps - e\subwscreps$
(see (\ref{eqn-active-epsilon})),
we have
$$
  y^T x_{\epsilon} = q^T {\Working}_{\epsilon} x_{\epsilon} =
  q^T(b\subwscreps - e\subwscreps).
$$ 

By assumption, $x_{\epsilon}$ is an optimal
vertex for the perturbed problem, so that
$y^T x_{\epsilon} \ge [\bineq]_j - \epsilon^j$ (since
constraint $j$ is an inequality).
But our further assumption (\ref{eqn-active-epsilon}) that the constraints in
$Y_{\epsilon}$ are active at $x_{\epsilon}$
for all sufficiently small $\epsilon$ implies that
this relation is an {\sl equality,\/} i.e.,
$y^T x_{\epsilon}  = [\bineq]_j  - \epsilon^j$.
Substituting $q^T (b\subwscreps - e\subwscreps)$
for $y^T x_{\epsilon}$ and rearranging, we obtain
$$
q^T b\subwscreps - [\bineq]_j  - q^T e\subwscreps + \epsilon^j = 0.
$$
The left-hand side of this relation is a polynomial
in $\epsilon$, in which 
$q^T b\subwscreps$ and $[\bineq]_j$ are independent of $\epsilon$
and the inner product $q^T e\subwscreps$ is a linear
combination of the distinct powers
of $\epsilon$ from (\ref{eqn-epspowers}), none of which
is equal to $j$, and there is a term $\epsilon^j$.
Such a polynomial can equal zero only
when $\epsilon$ is exactly equal to one of the polynomial's roots.
Hence equality cannot hold when $\epsilon$ is
allowed to be any arbitrarily small positive value, and we obtain
a contradiction.  The same argument applies for all the
constraints in $\Yscr_{\epsilon}$, so that
$y^T x_{\epsilon}  > [\bineq]_j  - \epsilon^j$ for all 
$j\in\Yscr_{\epsilon}$.  It follows that
$Y_{\epsilon}$ is empty and that only
the $n$ constraints in $\Wscr_{\epsilon}$ are active,
confirming that $x_{\epsilon}$ is a nondegenerate optimal vertex
with active set $\Aact_{\epsilon} = \Working_{\epsilon}$.

Letting $\lamact_{\epsilon}$ denote the necessarily unique
solution of $\Working_{\epsilon}^T \lamact_{\epsilon} = c$,
it follows from the ``only if''
direction of Proposition~\ref{prop-nondegenopt}
that the components of $\lamact_{\epsilon}$
corresponding to active inequality constraints are nonnegative:
\begin{equation}
\label{eqn-pertlamprops}
\Working_{\epsilon}^T \lamact_{\epsilon} = c
\quad\hbox{and}\quad
[\lamact_{\epsilon}]_i \ge 0
\;\;\hbox{when}\;\;
w_i \in\Iscract_{\epsilon}.
\end{equation}

\smallskip
\noindent
{\bf Part 2. Defining an optimal solution for the original problem.}
We now show that the working set $\Wscr_{\epsilon}$ for
the perturbed problem is
an {\sl optimal\/} working set for the original problem;
see Definition~\ref{def-optworking}.
Taking $\Wscr = \Wscr_{\epsilon} = \{w_1,\dots, w_n\}$
and $\Working = \Working_{\epsilon}$,
we define $\xstar$ as the (unique) solution of
$\Working\xstar = b\subwscr$, so that the $n$
linearly independent constraints represented in $\Working$
are active at $\xstar$.

Let $y^T$ denote the normal of any constraint not in
$\Working$, and assume that it corresponds to the
$j$th inequality constraint in the original problem.
It remains to show that $y^T \xstar \ge [\bineq]_j$,
i.e., that $\xstar$ is feasible with respect to the
corresponding original inequality constraint.
The proof of Part 1 shows that
there is a unique $q$ such that $y^T = q^T \Working$
and that
$q^T b\subwscr - [\bineq]_j  - q^T e\subwscr + \epsilon^j > 0$
for all sufficiently small $\epsilon$.
Since $y^T \xstar = q^T b\subwscr$, we have
\begin{equation}
\label{eqn-xstarstrict}
  y^T\xstar - [\bineq]_j - q^T e\subwscr + \epsilon^j > 0,
\end{equation}
which has two consequences.
\begin{enumerate}
\item[(i)]
A result from \cite[Lemma 1, Chapter 10]{Dantzig} says that
a polynomial in $\epsilon > 0$
will be positive for all sufficiently small $\epsilon$
if and only if the coefficient of the smallest power of
$\epsilon$ is positive. 
The cited result implies
that, if the constant term $y^T\xstar - [\bineq]_j$ of the
polynomial in (\ref{eqn-xstarstrict}) is nonzero,
it must be positive, in which case
original constraint $j$ is inactive at $\xstar$.
\item[(ii)] If $y^T\xstar - [\bineq]_j = 0$, then by
definition original constraint $j$ is
active at $\xstar$.  (This case applies when $\xstar$ is
degenerate.)
\end{enumerate}
In either case, $y^T \xstar \ge [\bineq]_j$ and $\xstar$ is
feasible with respect to all of the original inequality constraints
$\Aineq x\ge \bineq$.

The remaining ingredient needed to verify that
$\xstar$ and $\Wscr$ are optimal involves multipliers.  
Since the nonsingular working matrix $\Working$ has been
taken as $\Working_{\epsilon}$ and
$\Working^T \lamact_{\epsilon} = c$, we can define
an $m$-vector $\lamstar$, where $\lamstar\subwscr$ denotes
the vector of components of $\lamstar$ associated with
constraints in the working set:
\begin{equation}
\label{eqn-full-lamdef}
  \lamstar\subwscr = \lamact_{\epsilon}
\quad\hbox{and}\quad
  \lamstari = 0 \;\;\hbox{if}\;\;
  i\ne \Wscr,
\end{equation}
noting that $\lamstari \ge 0$ if
the associated constraint is an inequality in
$\Wscr$; see (\ref{eqn-pertlamprops}).
Thus we have obtained an optimal vertex $\xstar$
and an optimal working set.  Using the optimal working
set, a multiplier $\lamstar$ can be defined satisfying the
sufficient optimality conditions of Proposition~\ref{prop-suffone}.
\end{proof}

The just-completed proof shows that the
perturbed LP is guaranteed to have a nondegenerate optimal
vertex, but this vertex will in general depend on
the value of $\epsilon$ and
the ordering of the powers of $\epsilon$ in
the perturbed constraints.
This non-uniqueness of $x_{\epsilon}$
and $\Wscr_{\epsilon}$ is
illustrated in Figure~\ref{fig-degenpic} for the
all-inequality linear program (\ref{eqn-pertdegeninfig}).
The contours of the linear objective are labeled as 
``$\phi$''.
The optimal degenerate vertex $\xstar = (1,2)^T$ for
the original LP is shown on the left,
where the constraints include a thin shading
on the infeasible side. 
In the remaining two figures, the
constraints have been perturbed and the thickness of
the shading
reflects the size of the perturbation.
The value of $\epsilon$ is deliberately taken as
$\half$ so that the effects can easily be seen.
In the middle figure, constraints $1$, $2$, and $3$ are
perturbed respectively by $\epsilon$, $\epsilon^2$, and
$\epsilon^3$, producing a
single nondegenerate vertex where constraints $2$ and
$3$ are active.  In the rightmost figure,
the constraint perturbations are $\epsilon^2$, $\epsilon$,
and $\epsilon^3$, creating two distinct nondegenerate
vertices, with constraints $1$ and $3$ active at the optimal
vertex (which differs from the optimal vertex in the
middle figure).  Our earlier analysis of
(\ref{eqn-pertdegeninfig})
showed that the optimal working
sets are indeed $\Wscr_3 = \{2,3\}$ and $\Wscr_2 = \{1,3\}$,
shown respectively in the
middle and rightmost figures.


\begin{figure}[htb]
\label{fig-pertdegen}
\begin{center}
\centering
\centerline{\epsfbox{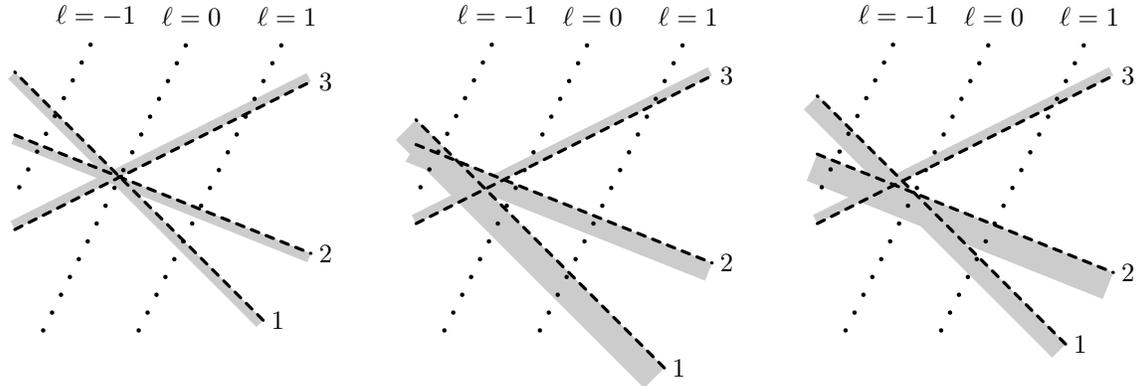}}
\parbox[t]{.9\textwidth}{
\caption{\label{fig-degenpic} \small
Effects of perturbing the constraints at a degenerate
optimal vertex.
}}
\end{center}
\end{figure}

\section{Necessary and sufficient optimality
conditions, version 2}

The result of Proposition~\ref{prop-findoptworking}
allows us to state necessary and sufficient
conditions for optimality of a linear
program with the form (\ref{eqn-togetherform})
in which the constraints are consistent, $\rank(A) = n$
(where $A$ is defined by (\ref{eqn-abfulldef})),
and the objective function is bounded below
in the feasible region.  Note that they apply at any
optimal point, whether or not it is a vertex.

\begin{proposition}[Necessary and sufficient optimality conditions---Version II.]
\label{prop-atlast}
Consider the linear program of minimizing $c^T x$ subject to
the consistent constraints $\Aeq x = \beq$ and
$\Aineq x \ge \bineq$, where
$c^T x$ is bounded below in the feasible region and
$\rank(A) = n$.
The point $\xtilde$, which need not be
a vertex, is optimal if and only if $\xtilde$ is
feasible
and there exists an $m$-vector $\lamtilde$ such that
$A^T \lamtilde = c$, $\lamtilde^T(A\xtilde - b) = 0$
and $\lamtilde\subiscr \ge 0$. The optimal objective value is $\lamtilde^T b$.
\end{proposition}
\begin{proof}
The ``if'' part was proved in Proposition~\ref{prop-suffone}.

To confirm the ``only if'' part, we begin by
observing that
Proposition~\ref{prop-findoptworking}
shows that an optimal vertex $\xstar$ must exist
for the given LP, with an
associated optimal working set $\Wscr$
that allows us to define
an optimal $m$-component multiplier $\lamstar$ such that
$A^T\lamstar = c$, $\lamstarT(A\xstar - b)$,
and $\lamstar\subiscr \ge 0$; see
(\ref{eqn-full-lamdef}).
An important point is that
$\lamstar\subiscr$ contains multipliers for all
the inequality constraints in the
problem. Proposition~\ref{prop-suffone} shows that the 
optimal value is $\lamstarT b$. 

Now suppose that the feasible point $\xtilde$
is optimal, where $\xtilde$
may or may not be a vertex. 
Proposition~\ref{prop-specialprops} states that
$\lamstarT(A\xtilde - b) = 0$ must hold. Hence we can take
$\lamtilde = \lamstar$ as a multiplier for $\xtilde$. Again, we stress
that optimality of $\xtilde$ follows 
directly from existence of the multiplier $\lamtilde$.
\end{proof}

\section{Identifying an optimal working set at an
optimal vertex}

The results proved thus far show that, when the
constraints $\Aeq  x = \beq$ and $\Aineq x \ge \bineq$ are
consistent, $\rank(A) = n$,
and $c^T x$ is bounded below in the feasible
region, an optimal vertex $\xhat$ and a corresponding optimal working
set $\Wscrhat$ of $n$ indices such that $\Workinghat$ is nonsingular must exist
(where the working set leads to
a multiplier $\lamhat$ that ensures optimality). 
We know from Proposition~\ref{prop-atlast}
that $\lamhat$ is also an optimal multiplier for any
optimal point $\xstar\ne\xhat$.
But the active constraints at $\xhat$ and $\xstar$
may be different, which means that the working set
$\Wscrhat$ may not be a valid working set
for $\xstar$ because
$\Workinghat\xstar\ne b\subwscrhat$; see
the example following the proof of 
Proposition~\ref{prop-optworking}.

The next result shows that given a specific
optimal vertex $\xstar$,
a corresponding optimal nonsingular working-set matrix $W$ of
of $n$ indices exists
satisfying Definition~\ref{def-optworking}.
The result relies on
Proposition~\ref{prop-specialprops}, which shows
that the multiplier associated with an optimal vertex
satisfies the sufficient optimality
conditions for any other optimal point.

\begin{proposition}
[Existence of an optimal working set.]
\label{prop-optworking}
For the linear program of minimizing $c^T x$ subject to
the consistent constraints $\Aeq x = \beq$ and
$\Aineq x\ge \bineq$, where
$\rank(A) = n$ and $c^T x$ is bounded below in
the feasible region, suppose that an
optimal vertex $\xstar$ is given.  Then there is an associated
optimal working set $\Wscr$ of $n$ indices, such that the
corresponding working-set matrix $W$ is nonsingular.
\end{proposition}
\begin{proof}
Proposition~\ref{prop-findoptworking} guarantees
existence of an optimal vertex $\xhat$,
an optimal working set $\Wscrhat$ containing $n$
indices such that the corresponding working-set matrix $\Workinghat$ is
nonsingular, and an $m$-dimensional optimal vector
$\lamhat$.
If $\xstar = \xhat$, we can take $\Wscr = \Wscrhat$
and nothing more is needed.
If $\xstar\ne \xhat$, we show next how to use
$\Wscrhat$ to construct an
optimal working set $\Wscr$ for $\xstar$.

The working set $\Wscrhat$ contains precisely $n$ indices, which
include those of the equality constraints plus a selection of active
inequality constraints.  Defining $\Wscrhat\subplus$ as the set of
indices of the equality constraints plus the indices $i$ of inequality
constraints with positive multipliers $\lamhat_i$, let
$\Workinghat\subplus$ denote the associated submatrix of
$\Workinghat$, i.e., the matrix whose rows correspond to indices in
$\Wscrhat\subplus$. Nonsingularity of $\Workinghat$ implies
that $\Workinghat\subplus$ has full row rank. Since $\xhat$ and
$\xstar$ are both optimal, we know from
Proposition~\ref{prop-specialprops} that complementarity is satisfied
for all constraints at both $\xhat$ and $\xstar$, which means that, if
an inequality constraint in $\Wscrhat$ has a positive multiplier, then
that constraint must be active at both $\xhat$ and $\xstar$. In
addition, all equality constraints are satisfied at both $\xhat$ and
$\xstar$. We therefore conclude that $\Workinghat\subplus$ is a
submatrix of $\Aact(\xstar)$. Defining $\Wscrhat_0$ as the set of
indices $i$ of inequality constraints that are active at $\xstar$ for
which $\lamhat_i=0$ and letting $\Workinghat_0$ denote the
corresponding matrix, it follows that
$$
\Aact(\xstar) = \mtx{cc}{\Workinghat\subplus \\
   \Workinghat_0}.
$$
Consequently, since $\xstar$ is a vertex, $\Aact(\xstar)$ has full
column rank. As $\Workinghat\subplus$ has full row rank, we may
therefore create a nonsingular $n\times n$ working-set matrix $W$ as a
nonsingular $n\times n$ submatrix of $\Aact(\xstar)$ that contains
$\Workinghat\subplus$ and let $\Wscr$ denote the associated indices.
\end{proof}

For example, consider a three-variable
all-inequality LP with six constraints
$Ax\ge b$, where
\begin{equation}
\label{eqn-workexamp}
  A = \mtx{rrr}{0 & 0 & 1\\
                1 & 2 & 1\\
                1 & -1 & 2\\
                1 & 1 & 1\\
                -1 & 0 & 1\\
                0 & 1 & -1},\quad
  b = \mtx{r}{1\\
              5\\
              3\\
              4\\
             -2\\[3pt]
              -\half},\quad\hbox{and}\quad
   c = \mtx{c}{1\\
            2\\
            3}.
\end{equation}
Two degenerate vertices, $\xstar = (2,1,1)^T$ and
$\xhat = (3,\half,1)^T$, are optimal, and the optimal
objective is $c^T\xstar = 7$.
Suppose that $\xhat$ is the optimal vertex produced by
Proposition~\ref{prop-findoptworking}.
The active set at $\xhat$ is
$\Ascr(\xhat) = \{1,2,5,6\}$, and $\Wscrhat = \{1,2,5\}$ 
is an optimal working set, with
$$
  \Workinghat\xhat = \mtx{rcc}{0 & 0 & 1\\
                           1 & 2 & 1\\
                          -1 & 0 & 1}
                     \mtx{c}{3\\
                            \half\\
                            1} = 
        b\subwscrhat = \mtx{r}{1\\
                            5\\
                           -2}
\quad\hbox{and}\quad
  \lamhat\subwscrhat = \mtx{c}{2\\
                            1\\
                            0}.
$$
The associated $6$-component optimal multiplier is
$\lamhat = (2,1,0,0,0,0)^T$.
Using the notation in the proof of
Proposition~\ref{prop-optworking},
$n\subplus = 2$ and $\Wscrhat\subplus = \{1,2\}$.

Now consider finding an optimal working set at $\xstar$.
As shown in the proof, constraints
$1$ and $2$ must be active at $\xstar$ (and
indeed they are), but constraints
$5$ and $6$ are not.  The active set at $\xstar$ is
$\Ascr(\xstar) = \{1,2,3,4\}$
and $\Workinghat\xstar \ne b\subwscrhat$,
so that
$\Wscrhat$ is not an optimal working set for $\xstar$,
even though $\lamhat$ is an optimal multiplier.

Constraints $1$ and $2$ must be part of the working set at $\xstar$.
Since $\rank(\Aact(\xstar)) = 3$, we need to add one further
constraint which is active at $\xstar$ to $\Wscrhat\subplus$.  For
this example, the extra constraint can be taken as constraint $3$ or
$4$.  In either case, the optimal multiplier is the same,
$\lamstar=\lamhat$, and the linear system $\Working \xstar =
b\subwscr$ is satisfied.

Note that if we seek a working set at a non-vertex optimal point,
such as $\xtilde = (\frac{5}{2}, \frac{3}{4}, 1)^T$ in
example (\ref{eqn-workexamp}), then 
$\Wscrtilde = \Wscrhat\subplus$ is an optimal working set
at $\xtilde$. In fact, $\Wscrhat\subplus$ is an optimal working set at
any optimal point.
 
\section{A proof of Farkas' lemma}
\label{sec-farkas}

For completeness, we state and prove a common form of Farkas' lemma
using the results in this paper. Note that in Farkas' lemma, no
requirement is imposed about the rank of the matrix involved.

\begin{lemma}[Farkas' lemma.]
\label{lem-farkas}
Given an $m\times n$ matrix $A$ and
an $n$-vector $c$,
precisely one of the following two conditions must be true:
\begin{enumerate}
\item[(1)] There exists $y\ge 0$ such that
$A\T y = c$;
\item[(2)] There exists $p$ such that
$A p\ge 0$ and $c^T p < 0$.
\end{enumerate}
\end{lemma}

\begin{proof}
If $y$ satisfies (1) and $p$ satisfies (2) then $c^T p = y^T A p$. Because
$A p\ge 0$ and $y\ge 0$, it follows that $y\T A p \ge 0$, which
contradicts the relation $c\T p<0$ in (2).  Hence (1) and (2) cannot both
be true.

To show that one of (1) or (2) must
be true, we consider the all-inequality linear program
\begin{equation}
\label{eqn-farkaslp}
\minimize{p\in\Re^n}\;\; c\T p 
\quad\hbox{subject to}\quad \Atilde p \ge b, \quad\hbox{with}\quad
\Atilde = \mtx{l}{\m A\\
            \m I_n\\
            -I_n}
\;\;\hbox{and}\;\;
b = \mtx{r}{0\\
            -e\\
            -e},
\end{equation}
where $e$ denotes $(1,1,\dots, 1)^T$.  The first $m$ constraints
are $Ap\ge 0$ and the last $2n$ constraints are equivalent
to requiring
that $-1 \le p_i \le 1$ for $i=1$, \dots, $n$.  

This LP has the following properties: (i) $\Atilde$ 
has rank $n$ because of the presence of the two identity matrices,
(ii) the constraints $\Atilde p\ge b$ are consistent because $p = 0$
is feasible, 
and (iii) the feasible region is bounded so the objective function is
bounded below. Let $\pstar$ denote an optimal solution of
(\ref{eqn-farkaslp}).

Proposition~\ref{prop-atlast} implies that there exists
a nonnegative optimal multiplier $\lambda$, which we may partition as
$\lambda=(\lambda_1,\lambda_2,\lambda_3)^T$, where $\lambda_1$ is an
$m$-vector and $\lambda_2$ and $\lambda_3$ are $n$-vectors.
Because $\lambda$ is an optimal multiplier,
we know that $\Atilde^T \lambda = c$ and
$c^T\pstar = \lambda^T b$.  Writing out these relations in
partitioned form gives
\begin{equation}\label{eqn-farkasopt} \\
\Atilde^T \lambda = A^T \lambda_1 + \lambda_2 - \lambda_3  = c
\;\;\hbox{with}\;\; \lambda_i\ge 0, 
\quad\hbox{and}\quad
  c^T \pstar = \lambda^T b = -e\T (\lambda_2 + \lambda_3).
\end{equation}
Since $c^T p = 0$ at the feasible point $p=0$, the optimal
objective value
$c^T\pstar$ must be either zero or negative.
If $c^T \pstar = 0$, the second relation in (\ref{eqn-farkasopt})
implies that $-e\T (\lambda_2 + \lambda_3)=0$. Since $\lambda_2$ and $\lambda_3$ are
both nonnegative, it follows that $\lambda_2=\lambda_3=0$. Consequently,
the first relation in (\ref{eqn-farkasopt}) shows that
$A^T \lambda_1=c$, $\lambda_1\ge 0$, which means that
that case (1) of the lemma holds for $y=\lambda_1$.
If $c^T \pstar < 0$, then, since $A\pstar \ge 0$, $\pstar$ satisfies
relation (2) of the lemma. Consequently, exactly one of (1) and (2) has a solution.
\end{proof}

\section{Summary}
\label{sec-summary}
Assume that the constraints $\Aeq x = \beq$
and $\Aineq x \ge \bineq$ are consistent,
$\rank(A) = n$, 
and $c^T x$ is bounded below in the feasible region.
We have shown that the feasible point $\xstar$ is
an optimal solution
if and only if there exists an optimal multiplier
$\lamstar$ such that (i) $A^T\lamstar = c$,
(ii) $\lamstarT(A\xstar -b) = 0$,
and (iii) $\lamstar\subiscr \ge 0$.
These conditions were derived through elementary
proofs of the following sequence of results:
\begin{enumerate}
\item[(a)] If
$\lamstar$ exists satisfying (i), (ii), and (iii),
$\xstar$ is optimal. (Proposition~\ref{prop-suffone}.)
\item[(b)] Let $\xstar$ be an optimal point with an
associated multiplier $\lamstar$ satisfying
(i), (ii), and (iii).  For any other optimal point
$\xtilde\ne\xstar$, condition (ii) is satisfied with
$\lamstar$ and
$\xtilde$, i.e., $\lamstarT(A\xtilde - b) = 0$,
and $c^T \xtilde = c^T\xstar$.
(Proposition~\ref{prop-specialprops}.)
\item[(c)] There always exists an optimal vertex $\xhat$
and an associated multiplier $\lamhat$ satisfying
(i), (ii), and (iii).
(Propositions~\ref{prop-nondegenopt} and
\ref{prop-findoptworking}.)
\item[(d)] There must be an optimal multiplier
corresponding to any optimal solution.
(Proposition~\ref{prop-atlast}.)
\item[(e)] There must be a nonsingular optimal working set corresponding
to any optimal vertex. (Proposition~\ref{prop-optworking}.)
\end{enumerate}

\end{document}